\documentclass{amsproc}

\usepackage{euler, epic,eepic,latexsym, amssymb, amscd, amsfonts, xypic, url, color, epsfig}

\newtheorem{theorem}{Theorem}[section]
\newtheorem{lemma}[theorem]{Lemma}

\theoremstyle{definition}

\theoremstyle{remark}
\newtheorem{remark}[theorem]{Remark}

\numberwithin{equation}{section}

\newcommand{\hidden}[1]{\footnote{Hidden:  #1}}
\renewcommand{\hidden}[1]{}

\newcommand{\Z}{\mathbb{Z}}

\newcommand{\N}{\mathbb{N}}

\newcommand{\W}{\mathbb{W}}

\newcommand{\Ahat}{\hat{\mathbb{A}}}

\newcommand{\Hom}{\operatorname{Hom}}

\newcommand{\Gr}{\operatorname{Gr}}
\newcommand{\fs}{\text{fs}}
\newcommand{\fg}{\text{fg}}
\newcommand{\topHopf}{\text{top Hopf alg}}
\newcommand{\Hopf}{\text{Hopf alg}}
\newcommand{\topalg}{\text{top alg}}
\newcommand{\alg}{\text{alg}}
\newcommand{\coalg}{\text{coalg}}
\newcommand{\Spf}{\operatorname{Spf}}

\newcommand{\Mor}{{\text{Mor}}}

\newcommand{\rK}{{\rm K}}

\begin{document}

\title{Cartier's first theorem for Witt vectors on $\Z_{\geq 0}^n - 0$}


\author{Kirsten Wickelgren}
\address{Dept. of Mathematics, Harvard University, Cambridge~MA}
\curraddr{School of Mathematics,
Georgia Institute of Technology, Atlanta, Georgia 30332}
\email{wickelgren@post.harvard.edu}
\thanks{Supported by an American Institute of Mathematics Five Year Fellowship.}



\date{August 8, 2012 and, in revised form, November 14, 2013.}

\begin{abstract}
We show that the dual of the Witt vectors on $\Z_{\geq 0}^n - 0$ as defined by Angeltveit, Gerhardt, Hill, and Lindenstrauss represent the functor taking a commutative formal group $G$ to the maps of formal schemes $\Ahat^n \to G$, and that the Witt vectors are self-dual for $\mathbb{Q}$-algebras or when $n=1$.
\end{abstract}

\maketitle

\section{Introduction}

Hesselholt and Madsen computed the relative $\rK$-theory of $k[x]/\langle x^a \rangle$ for $k$ a perfect field of positive characteristic in \cite{HM}, and give the answer in terms of the Witt vectors of $k$. In the analogous computation for the ring $$A= k[x_1, \ldots, x_n]/\langle x_1^{a_1}, \ldots, x_n^{a_n} \rangle,$$ Angeltveit, Gerhardt, Hill, and Lindenstrauss define an $n$-dimensional version of the Witt vectors, which they use to express the relative $\rK$-theory and topological cyclic homology of $A$ \cite{AGHL}. 

We show that the Cartier dual of the additive group underlying their Witt vectors on the truncation set $\Z_{\geq 0}^n - 0$, denoted $\W_{\Z_{\geq 0}^n - 0}$, represents the functor taking a commutative formal group $G$ to the pointed maps of formal schemes $\Ahat^n \to G$ (Theorem \ref{CFT_dual_analogue}). We also show that the additive group of $\W_{\Z_{\geq 0}^n - 0}$ is self dual (Lemma \ref{Wn=Wnstar}) when $n=1$ or $R$ is a $\mathbb{Q}$-algebra. Combining these results implies that the additive formal group of $\W_{\Z_{\geq 0}^n - 0}$ represents the functor sending $G$ to the group of maps $\Ahat^n \to G$ when $n=1$ or $R$ is a $\mathbb{Q}$-algebra. The case of $n=1$ is Cartier's first theorem \cite{C} \cite[Th. 27.1.14]{H} on the classical Witt vectors. 


{\bf Acknowledgements.} This paper is a result of Teena Gerhardt's lovely talk at the Stanford Symposium on Algebraic Topology, and a discussion with Michael Hopkins about it. There are ideas of Hopkins in this paper, and I warmly thank him for them. I also thank Gerhardt, Michael Hill, Joseph Rabinoff, and the referee for useful comments. As always, my gratitude and admiration for Gunnar Carlsson are difficult to express.

\section{Cartier's first theorem for Witt vectors on $\Z_{\geq 0}^n - 0$}

Here is Angeltveit, Gerhardt, Hill, and Lindenstrauss's $n$-dimensional version of the Witt vectors, defined in Section $2$ of \cite{AGHL}: a set $S \subseteq \Z_{\geq 0}^n - 0$ is a {\em truncation set} if $(k j_1, k j_2, \ldots, k j_n)$ in $S$ for $k \in \N = \Z_{>0}$ implies that $(j_1, j_2, \ldots, j_n)$ is in $S$. For $\vec{J} = (j_1, \ldots, j_n)$ in $ \Z_{\geq 0}^n - 0$, let $\gcd(\vec{J})$ denote the greatest common divisor of the non-zero $j_i$. Given a ring $R$ and a truncation set $S$, let the Witt vectors $\W_S(R)$ be the ring with underlying set $R^S$ and addition and multiplication defined so that the ghost map $$\W_S(R) \to R^S $$ that takes $\{r_{\vec{I}} : \vec{I} \in S \}$ to $\{w_{\vec{I}}: \vec{I} \in S \}$ where $$ w_{\vec{I}} = \sum_{k \vec{J} = \vec{I}} \gcd(\vec{J}) r_{\vec{J}}^k$$ is a ring homomorphism, where in the above sum, $k$ ranges over $\N$ and $\vec{J}$ is in $S$. In \cite{AGHL}, one requires $S$ to be a subset of $\N^n$, but the same proof  that there is a unique functorial way to define such a ring structure \cite[Lem 2.3]{AGHL} holds for $S \subseteq \Z_{\geq 0}^n - 0$. Note that $$\W_S(R) = \prod_{Z \subsetneq \{1,\ldots n\}} \W_{S_Z}(R)$$ where $S_Z$ is defined $S_Z = \{(j_1, \ldots, j_n) \in S : j_i = 0\textrm{ if and only if } i \in Z \}$, and that for $S = \Z_{\geq 0}^n - 0$, we have $\W_{S_Z}(R) \cong \W_{\N^m}(R)$ with $m = n - \vert Z \vert$.

Let $R$ be a ring. For any truncation set $S$, the additive group underlying the ring $\W_S(R)$ determines a commutative group scheme and formal group over $R$. 

Let $\Ahat^n = \Spf R [[t_1, t_2,\ldots, t_n]]$ be formal affine $n$-space and consider $\Ahat^n$ as a pointed formal scheme, equipped with the point $\Spf R \to \Ahat^n$ corresponding to the ideal $\langle t_1, \ldots t_n \rangle$. Let $\Mor_{\fs}(\Ahat^n, G)$ denote the morphisms of pointed formal schemes over $R$ from $\Ahat^n$ to a pointed formal $R$-scheme $G$. The identity of a formal group $G$ gives $G$ the structure of a pointed formal scheme.

For commutative formal groups $G_1$ and $G_2$ over $R$, let $\Mor_{\fg}(G_1, G_2)$ denote the corresponding morphisms. 

\begin{theorem}\label{CFTanalogue}Suppose $R$ is a $\mathbb{Q}$-algebra or $n=1$. The additive formal group of $\W_{\Z_{\geq 0}^n - 0}(R)$ represents the functor$$G \mapsto \Mor_{\fs}(\Ahat^n,G)$$  from commutative formal groups over $R$ to groups, i.e. there is a natural identification $$\Mor_{\fg}(\W_{\Z_{\geq 0}^n - 0}(R), G) \cong \Mor_{\fs}(\Ahat^n,G)$$ for commutative formal groups $G$ over $R$.\end{theorem}

Theorem \ref{CFTanalogue} is proven by combining Theorem \ref{CFT_dual_analogue} and Lemma \ref{fgW_as_topHopf} below.

Cartier duality gives a contravariant equivalence between certain topological $R$-algebras and $R$-coalgebras \cite[Prop 37.2.7]{H}. For such a topological $R$-algebra (respectively coalgebra) $B$, let $B^{\star}$ denote its Cartier dual $$B^{\star} = \Mor_R(B, R) $$ where $\Mor_R(B,R)$ denotes the continuous $R$-module homomorphisms from $B$ to $R$ (respectively the $R$-module homomorphisms from $B$ to $R$). Say that an algebra or coalgebra is {\em augmented} if it is equipped with a splitting of the unit or counit map. It is straightforward to see that Cartier duality induces an equivalence between augmented topological $R$-algebras satisfying the conditions of \cite[37.2.4]{H} and augmented $R$-coalgebras satisfying the conditions of \cite[37.2.5]{H}. Denote the morphisms in the former category by $\Mor_{\topalg}(-,-)$ and the morphisms in the latter category by $\Mor_{\coalg}(-,-)$.

The commutative group scheme determined by the additive group underlying $\W_S(R)$ has a Cartier dual $\W_S(R)^{\star}$ which is a topological Hopf algebra or formal group.

\begin{theorem}\label{CFT_dual_analogue}The Cartier dual of the additive group scheme of $\W_{\Z_{\geq 0}^n - 0}(R)$ represents the functor$$G \mapsto \Mor_{\fs}(\Ahat^n,G)$$  from commutative formal groups over $R$ to groups, i.e. there is a natural identification $$\Mor_{\fg}(\W_{\Z_{\geq 0}^n - 0}(R)^{\star}, G) \cong \Mor_{\fs}(\Ahat^n,G)$$ for commutative formal groups $G$ over $R$.\end{theorem}

\begin{proof}
First assume that the formal group $G$ is affine. Let $A$ denote the functions of $G$, so $A$ is a Hopf algebra and $G= \Spf A$. $$\Mor_{\fs}(\Ahat^n, G) = \Mor_{\topalg}(A, R[[t_1, t_2, \ldots, t_n]]).$$\hidden{ where the right hand side denotes the continuous $R$-algebra morphisms respecting the augmentations $R[[t_1, t_2, \ldots, t_n]] \to R$ and $A \to R$ from the distinguished points on the corresponding formal schemes.} 

By Cartier duality, $$\Mor_{\topalg}(A, R[[t_1, t_2, \ldots, t_n]]) = \Mor_{\coalg}(R[[t_1, t_2, \ldots, t_n]]^{\star}, A^{\star}).$$

Let $F$ denote the left adjoint to the functor taking a Hopf algebra (as defined \cite[37.1.7]{H}) to its underlying augmented coalgebra. Since $A$ is a Hopf algebra, so is $A^{\star}$. Therefore,  \begin{align*}& \Mor_{\coalg}(R[[t_1, t_2, \ldots, t_n]]^{\star}, A^{\star}) = \Mor_{\Hopf}(F(R[[t_1, t_2, \ldots, t_n]]^{\star}), A^{\star})= \\  & \Mor_{\topHopf}(A,F(R[[t_1, t_2, \ldots, t_n]]^{\star})^{\star} ) = \Mor_{fg}(\Spf F(R[[t_1, t_2, \ldots, t_n]]^{\star})^{\star}, G),\end{align*} where $\Mor_{\topHopf}(-,-)$ denotes morphisms of topological Hopf algebras whose underlying topological $R$-algebra is as before.


By Lemma \ref{Wn=Fstar} proven below, the formal group $\Spf F(R[[t_1, t_2, \ldots, t_n]]^{\star})^{\star}$ is isomorphic to the Cartier dual of the additive group scheme of $\W_{\Z_{\geq 0}^n - 0}(R)$.

Thus $\W_{\Z_{\geq 0}^n - 0}(R)^{\star}$ represents the functor $G \mapsto \Mor_{\fs}(\Ahat^n,G)$ restricted to affine commutative formal groups $G$. Since $\W_{\Z_{\geq 0}^n - 0}(R)^{\star}$ is an affine formal group, the identity morphism determines an element of $\Mor_{fs}(\Ahat^n,\W_{\Z_{\geq 0}^n - 0}(R)^{\star})$, which in turn defines a natural transformation $$\eta: \Mor_{\fg}(\W_{\Z_{\geq 0}^n - 0}(R)^{\star}, -) \to \Mor_{\fs}(\Ahat^n,-) .$$ For any formal group $G$, the sets $\Mor_{\fg}(\W_{\Z_{\geq 0}^n - 0}(R)^{\star}, G)$ and $\Mor_{\fs}(\Ahat^n,G)$ extend to sheaves on $\Spf R$. Since locally on $\Spf R$, every formal group $G$ is affine, $\eta$ is a natural isomorphism.

\end{proof}

\begin{lemma}\label{Wn=Fstar} The group scheme determined by the Hopf algebra $$F(R[[t_1, t_2, \ldots, t_n]]^{\star})$$ is isomorphic to the additive group scheme of $\W_{\Z_{\geq 0}^n - 0}(R)$.\end{lemma}

\begin{proof}
For notational convenience, given $\vec{I}=(i_1,i_2,\ldots,i_n)$ and $\vec{J}=(j_1,\ldots,j_n)$ in $\Z_{\geq 0}^n$, let $t^{\vec{I}} = t_1^{i_1}t_2^{i_2} \cdots t_n^{i_n}$, and write $\vec{I} \leq \vec{J}$ when $i_k \leq j_k$ for all $k$.

$R[[t_1, t_2, \ldots, t_n]]^{\star}$ is a free $R$-module on the basis $\{ b_{\vec{I}} : \vec{I}=(i_1,i_2,\ldots,i_n) \in \Z_{\geq 0}^n\}$ where $b_{\vec{I}}$ is dual to $t_1^{i_1}t_2^{i_2} \cdots t_n^{i_n}$.  The $R$-coalgebra structure is given by\hidden{$$\sum_{(i_1,i_2,\ldots,i_n) \in \Z_{\geq 0}^n} a_{i_1,i_2,\ldots,i_n}b_{i_1,i_2,\ldots,i_n} \mapsto a_{0,0,\ldots,0}$$ and } the comultiplication \begin{equation}\label{R[t]starcomult} b_{\vec{I}} \mapsto \sum_{0 \leq \vec{J} \leq \vec{I}} b_{\vec{J}} \otimes b_{\vec{I} - \vec{J}},\end{equation} and the augmentation $R \to R[[t_1, t_2, \ldots, t_n]]^{\star}$ sends $r$ to $rb_{\vec{0}}$. 

It follows that $F(R[[t_1, t_2, \ldots, t_n]]^{\star})$ is the polynomial algebra $$R[ b_{\vec{I}} : \vec{I} \in \Z_{\geq 0}^n ] / \langle b_{\vec{0}} - 1\rangle$$ with comultiplication equal to the $R$-algebra morphism determined by \eqref{R[t]starcomult}. Thus, for any $R$-algebra $B$ $$\Mor_{\alg}(F(R[[t_1, t_2, \ldots, t_n]]^{\star}), B)$$ is the group under multiplication of power series in $n$ variables $t_1,t_2,\ldots,t_n$ with leading coefficient $1$ and coefficients in $B$ \begin{equation}\label{Bpwr_series}\{ 1+ \sum_{\vec{I} \in \Z_{\geq 0}^n- 0} b_{\vec{I}} t^{\vec{I}}: b_{\vec{I}} \in B \}.\end{equation}

Any such power series can be written uniquely in the form \begin{equation}\label{aipoly} \prod_{\vec{I} \in \Z_{\geq 0}^n - 0} (1 - a_{\vec{I}} t^{\vec{I}})\end{equation} with $a_{\vec{I}} \in B$. It follows that $F(R[[t_1, t_2, \ldots, t_n]]^{\star})$ is isomorphic as a Hopf algebra to the polynomial algebra $R[a_{\vec{I}} : \vec{I} \in \Z_{\geq 0}^n - 0]$ with comultiplication determined by multiplication of power series of the form \eqref{aipoly}. By the definition of the Witt vectors, it suffices to show that the Witt polynomials $ \sum_{k \vec{J} = \vec{I}} \gcd(\vec{J}) a_{\vec{J}}^k$ are primitives for this comultiplication for all $\vec{I}$ in $\Z_{\geq 0}^n- 0$. To show this, we may assume that $R$ is a free ring, since every ring is a quotient of a free ring. Then $R$ embeds into its field of fractions, so we may further assume that $k$ is invertible for all $k \in \Z_{> 0}$. Note that \begin{align*}\log \prod_{\vec{I} \in \Z_{\geq 0}^n - 0} (1 - a_{\vec{I}} t^{\vec{I}}) & = - \sum_{\vec{I} \in \Z_{\geq 0}^n- 0} \sum_{k \in \N}  \frac{a_{\vec{I}}^k}{k} t^{k\vec{I}} \\ & = - \sum_{\vec{I} \in \Z_{\geq 0}^n- 0} \sum_{k \vec{J} = \vec{I}} \frac{a_{\vec{J}}^k}{k}  t^{\vec{I}} \\ & = \sum_{\vec{I} \in \Z_{\geq 0}^n- 0} \big(\sum_{k \vec{J} = \vec{I}} \gcd(\vec{J}) a_{\vec{J}}^k \big) ~\frac{- t^{\vec{I}}}{\gcd(\vec{I})}.\end{align*}

Thus the group under multiplication with elements \eqref{aipoly} is isomorphic to the group with elements $\{ a_{\vec{I}} \in B\} $ and whose group operation is such that $$ \big(\sum_{k \vec{J} = \vec{I}} \gcd(\vec{J}) a_{\vec{J}}^k \big)$$ is an additive homomorphism, i.e. the Witt polynomials $ \sum_{k \vec{J} = \vec{I}} \gcd(\vec{J}) a_{\vec{J}}^k$ are indeed primitives as desired.
\end{proof}

The additive group scheme of $\W_{\Z_{\geq 0}^n - 0}(R)$ corresponds to a {\em graded} Hopf algebra, meaning that there is a grading on the underlying $R$-module such that the structure maps are maps of graded $R$-modules. This grading can be defined by giving $a_{\vec{J}}$ as in Lemma \ref{Wn=Fstar} degree $j_1+j_2 + \ldots + j_n$.  A graded Hopf algebra $B$ whose underlying graded $R$-module is free and finite rank in each degree has a graded Hopf algebra dual $B^*$ which we define to have $m$th graded piece $\Gr_m B^* = \Hom_R(\Gr_m B,R)$ and $$B^* = \oplus_m  \Gr_m B^*.$$ Note the difference with the Cartier dual $$B^{\star} = \prod_m  \Gr_m B^*.$$ Say that a graded Hopf algebra $B$ is {\em self dual} if there is an isomorphism $B \cong B^*$. An affine group scheme corresponding to a graded Hopf algebra will be called {\em self dual} if its corresponding graded Hopf algebra is self dual.

\begin{lemma}\label{Wn=Wnstar}The graded additive group scheme of $\W_{\Z_{\geq 0}^n - 0}(R)$ is self dual if $R$ is a $\mathbb{Q}$-algebra or if $n=1$. \end{lemma}

\begin{proof}
We give an isomorphism of graded Hopf algebras $$F(R[[t_1, t_2, \ldots, t_n]]^{\star}) \cong F(R[[t_1, t_2, \ldots, t_n]]^{\star})^*$$ which is equivalent to the claim by Lemma \ref{Wn=Fstar}. 

We saw above that $F(R[[t_1, t_2, \ldots, t_n]]^{\star})$ is the polynomial algebra $$R[ b_{\vec{I}} : \vec{I} \in \Z_{\geq 0}^n ] / \langle b_{\vec{0}} - 1\rangle$$ with comultiplication determined by \eqref{R[t]starcomult}. Thus, an $R$-basis for $F(R[[t_1, t_2, \ldots, t_n]]^{\star})$ is given by the collection of monomials $ b_{\vec{I}_1}^{m_1} b_{\vec{I}_2}^{m_2} b_{\vec{I}_3}^{m_3} \cdots b_{\vec{I}_k}^{m_k} $ in the variables $\{ b_{\vec{I}} : \vec{I} \in \Z_{\geq 0}^n - 0 \}$.  Let $\mathcal{C} = \{ c_{\vec{I}_1^{m_1} \vec{I}_2^{m_2} \vec{I}_3^{m_3} \cdots \vec{I}_k^{m_k} } : m_j > 0,  \vec{I}_j \in \Z_{\geq 0}^n - 0\}$ denote the dual basis of $F(R[[t_1, t_2, \ldots, t_n]]^{\star})^*$. For notational convenience, we will also write $c_{\vec{I}_1^{m_1} \vec{I}_2^{m_2} \vec{I}_3^{m_3} \cdots \vec{I}_k^{m_k} }$ even when some of the $m_j$ are $0$; it is to be understood that such an expression is identified with the corresponding expression with the $\vec{I}_j^{m_j}$ terms with $m_j=0$ removed.

Let $\{ e_1, e_2, \ldots, e_n \}$ be the standard basis of $\Z^n$, so $e_1 = (1,0,0,\ldots,0)$, $e_2 = (0,1,0,\ldots,0)$ etc.  For notational convenience, for $\vec{M} = (m_1, m_2, \ldots, m_n)$ in $\Z_{\geq 0}^n - 0$, let $C_{\vec{M}}$ abbreviate $c_{e_1^{m_1} e_2^{m_2} \cdots e_n^{m_n}}$. 

Note that $$\mu(C_{\vec{M}}) =  \sum_{0 \leq \vec{J} \leq \vec{M}} C_{\vec{J}} \otimes C_{\vec{M}- \vec{J}}  $$ where $\mu$ denotes the comultiplication of $F(R[[t_1, t_2, \ldots, t_n]]^{\star})^*$.

Sending $b_{\vec{I}}$ to $C_{\vec{I}}$ thus defines a morphism of Hopf algebras $$F(R[[t_1, t_2, \ldots, t_n]]^{\star}) \to F(R[[t_1, t_2, \ldots, t_n]]^{\star})^*,$$ and to prove the lemma it suffices to see that the $C_{\vec{I}}$ are free $R$-algebra generators of $F(R[[t_1, t_2, \ldots, t_n]]^{\star})^*$ when either $n =1$ or $\mathbb{Q} \subseteq R$.

We first show that the $C_{\vec{I}}$ generate $F(R[[t_1, t_2, \ldots, t_n]]^{\star})^*$ as an $R$-algebra in both cases:

First assume that $n=1$. We show that the $C_{m} = c_{e_1^m}$ for $m = 1,2,3, \ldots$ generate $F(R[[t_1]]^{\star})^*$ as an $R$-algebra. An arbitrary element $c$ of $\mathcal{C}$ is of the form $c_{i_1, i_2, \cdots, i_k}$ with the $i_k$ not necessarily distinct in $\Z_{> 0}$. Define the degree of $c$ to be $d=\sum_{j = 1}^k i_j$. Assume by induction that any element of $\mathcal{C}$ of degree less than $d$ is in the subalgebra generated by the $C_m$. Define the length of $c$ to be $k$. The length of $c$ must be less than or equal to $d$. If the length of $c$ equals $d$, then $c_{i_1, i_2, \cdots, i_k} = C_k$ and $c$ is in the subalgebra. So we may assume by induction that any element of $\mathcal{C}$ of degree $d$ and length greater than $k$ is in the subalgebra. The multiplication on $F(R[[t_1]]^{\star})^*$ is dual to $$b_{i_1} b_{i_2} b_{i_3} \cdots b_{i_k} \mapsto \prod_{j = 1}^k (\sum_{0 \leq J \leq i_j} b_{J} \otimes b_{i_j - J} ).$$ Thus the difference $$c -  c_{i_1-1, i_2-1, \cdots, i_k-1} C_k $$ is a sum of terms of degree $d$ and length greater than $k$. It follows by induction that the $C_{m} = c_{e_1^m}$ generate $F(R[[t_1]]^{\star})^*$ as claimed.

Now let $n$ be arbitrary. Consider the map $f: R [[t]] \to R[[t_1,\ldots,t_n]] $ defined by $$f(t) = t_1 + t_2 + \ldots + t_n.$$ There is an induced map $$f:  F(R[[t_1,\ldots,t_n]]^{\star}) \cong R[ b_{\vec{I}} : \vec{I} \in \Z_{\geq 0}^n-0 ]  \to R[b_m : m \in \Z_{> 0}] \cong F(R[[t]]^{\star})$$ which is determined by the following calculation of $f(b_{\vec{I}})$ for $\vec{I} = (i_1, i_2, \ldots, i_n)$. \begin{align*} & f(b_{\vec{I}}) (t^m) = b_{\vec{I}} (f(t^m)) = b_{\vec{I}} (t_1 + \ldots + t_n)^m = \\ & b_{\vec{I}} \big(\sum_{a_1, \ldots a_n \geq 0} { m \choose a_1 a_2 \cdots a_n } t_1^{a_1} t_2^{a_2} \cdots t_n^{a_n} \big),\end{align*} where the sum runs over non-negative $a_i$ whose sum is $m$ and where $${ m \choose a_1 a_2 \cdots a_n } = \frac{m!}{ a_1! a_2! \cdots a_n!}.$$ Thus $$f(b_{\vec{I}}) = {d \choose i_1 i_2 \cdots i_n} b_d,$$ where $d = \sum_{j = 1}^n i_j$. There is likewise an induced map $$f: F(R[[t]]^{\star})^*\to F(R[[t_1, t_2, \ldots, t_n]]^{\star})^*,$$ and we identify $$F(R[[t]]^{\star})^*\cong R[c_{i_1, i_2, \cdots, i_k}: i_j \in \Z_{> 0}]$$ and $$F(R[[t_1, t_2, \ldots, t_n]]^{\star})^* \cong  R[c_{\vec{I}_1 \vec{I}_2 \cdots \vec{I}_k}: \vec{I}_j \in \Z^n_{\geq 0} - 0].$$ By calculation as above, this map satisfies $$f(C_m) = \sum_{\operatorname{degree} \vec{I}= m} C_{\vec{I}}$$ where the sum runs over $\vec{I} \in \Z^n_{\geq 0}$ of degree $m$, and $$f(c_m) = \sum_{\operatorname{degree} \vec{I}= m} {m \choose \vec{I}}c_{\vec{I}},$$ where  $${m \choose \vec{I}}= {m \choose i_1 i_2 \ldots i_n}$$ when $\vec{I} = (i_1, i_2, \ldots, i_n)$. By the $n=1$ case, $f(c_m)$ is in the $R$-subalgebra generated by the $f(C_m)$. Since $F(R[[t_1, t_2, \ldots, t_n]]^{\star})^*$ is a $\Z^n$-graded Hopf algebra, it follows that the homogenous pieces of  $f(c_m)$ are in the $R$-subalgebra generated by the homogeneous pieces of $f(C_m)$. Thus ${m \choose \vec{I}}c_{\vec{I}}$ is in the $R$-subalgebra generated by the $C_{\vec{I}}$. Since ${m \choose \vec{I}}$ is invertible in $R$, it follows that $c_{\vec{I}}$ is in this subalgebra. 

An arbitrary element $c$ of $\mathcal{C}$ is of the form $c_{\vec{I}_1 \vec{I}_2 \cdots \vec{I}_k}$. The multiplication on $F(R[[t_1, t_2, \ldots, t_n]]^{\star})^*$ is dual to $$b_{\vec{I}_1} b_{\vec{I}_2} b_{\vec{I}_3} \cdots b_{\vec{I}_k} \mapsto \prod_{j = 1}^k (\sum_{0 \leq \vec{J} \leq \vec{I}_j} b_{\vec{J}} \otimes b_{\vec{I}_j - \vec{J}} ).$$ It follows that the difference $c - c_{\vec{I}_1 \vec{I}_2 \cdots \vec{I}_{k-1}} c_{\vec{I}_k} $ is a linear combination of elements of $\mathcal{C}$ of length less than $k$. Thus $c$ is in the $R$-subalgebra generated by the $C_{\vec{I}}$ by induction on the length $k$.




We now show that there are no relations among the $C_{\vec{I}}$, i.e. that the distinct monomials $C_{\vec{I}_1}C_{\vec{I}_2} \cdots C_{\vec{I}_k}$ form an $R$-linearly independent subset of $$F(R[[t_1, t_2, \ldots, t_n]]^{\star})^*:$$ 

Fix $\vec{M}$ in $\Z_{\geq 0}^n - 0$.  Let $\mathcal{I}$ denote the set of finite sets $\{ \vec{I}_1, \vec{I}_2, \ldots , \vec{I}_k\}$ with $\vec{I}_j$ in $\Z_{\geq 0}^n - 0$ and $\sum_{j = 1}^k \vec{I}_ j = \vec{M}$. For $S$ in $\mathcal{I}$ with $S = \{ \vec{I}_1, \vec{I}_2, \ldots , \vec{I}_k\}$, let $C_S = \prod_{j = 1}^k C_{\vec{I}_j}$ in $F(R[[t_1, t_2, \ldots, t_n]]^{\star})^*$ and let $c_S = c_{\vec{I}_1 \vec{I}_2 \cdots \vec{I}_k}$ in $\mathcal{C}$. Note that for all $S$ in $\mathcal{I}$, $C_S$ is in the sub-$R$-module $\mathcal{F}_{\vec{M}}$ spanned by $\{ c_S : S \in \mathcal{I} \}$. By the above, $\{ C_S: S \in \mathcal{I}\}$ spans $\mathcal{F}_{\vec{M}}$. Since $\mathcal{F}_{\vec{M}}$ is isomorphic to $R^N$ where $N$ is the (finite) cardinality of $\mathcal{I}$, any spanning set of size $N$ is also a basis \cite[Ch 3 Exercise 15]{AM}. In particular $\{ C_S : S \in \mathcal{I} \}$ is an $R$-linearly independent set. Since any monomial in the $C_{\vec{I}}$ is of the form $C_S$ for some $\vec{M}$, it follows that the distinct monomials in the $C_{\vec{I}}$ form a linearly independent set. \end{proof}

\begin{remark} The $C_{\vec{I}}$ do not generate $F(R[[t_1, t_2, t_3]]^{\star})^*$ when $2$ is not invertible in $R$ as can be checked by computing that the homogenous degree-$(1,1,1)$ component of the $R$-subalgebra generated by the $C_{\vec{I}}$ is the span of the following five vectors \begin{align*} &C_{e_1} C_{e_2} C_{e_3}=c_{(1,1,1)} + c_{(1,1,0)(0,0,1)}+ c_{(1,0,1)(0,1,0)}+  c_{(0,1,1)(1,0,0)}+ c_{e_1e_2e_3}, \\ &C_{(0,1,1)}C_{e_1} = c_{(1,1,0)(0,0,1)}+c_{(1,0,1)(0,1,0)} +c_{e_1e_2e_3}\\ &C_{(1,0,1)}C_{e_2} = c_{(0,1,1)(1,0,0)} +  c_{(1,1,0)(0,0,1)}+c_{e_1e_2e_3} \\ &C_{(1,1,0)}C_{e_3} =  c_{(1,0,1)(0,1,0)}+  c_{(0,1,1)(1,0,0)} + c_{e_1e_2e_3} \\ &C_{(1,1,1)} = c_{e_1e_2e_3}.\end{align*} \end{remark}

\begin{lemma}\label{fgW_as_topHopf}If $R$ is a $\mathbb{Q}$-algebra or if $n=1$, the Cartier dual of the additive group scheme of $\W_{\Z_{\geq 0}^n - 0}(R)$ is the formal group associated to the additive group of $\W_{\Z_{\geq 0}^n - 0}(R)$.\end{lemma}

\begin{proof}
By Lemma \ref{Wn=Fstar}, the claim is equivalent to showing that the topological Hopf algebra $F(R[[t_1, t_2, \ldots, t_n]]^{\star})^{\star}$ is the ring of functions of the formal group associated to the additive group of $\W_{\Z_{\geq 0}^n - 0}(R)$.
 
The Cartier dual $F(R[[t_1, t_2, \ldots, t_n]]^{\star})^{\star}$ of the Hopf algebra $F(R[[t_1, t_2, \ldots, t_n]]^{\star})$ is the product $$ F(R[[t_1, t_2, \ldots, t_n]]^{\star})^{\star} \cong \prod_{m=0}^{\infty} \Gr_m F(R[[t_1, t_2, \ldots, t_n]]^{\star})^*$$ over $m$ of the $m$th graded pieces of the graded Hopf algebra dual. By Lemma \ref{Wn=Wnstar}, $$F(R[[t_1, t_2, \ldots, t_n]]^{\star})^* \cong F(R[[t_1, t_2, \ldots, t_n]]^{\star}) \cong R[ b_{\vec{I}} : \vec{I} \in \Z_{\geq 0}^n ] / \langle b_{\vec{0}} - 1\rangle,$$ with comultiplication determined by \eqref{R[t]starcomult}. So $$\prod_{m=0}^{\infty} \Gr_m F(R[[t_1, t_2, \ldots, t_n]]^{\star})^* \cong R[[ b_{\vec{I}} : \vec{I} \in \Z_{\geq 0}^n ]] / \langle b_{\vec{0}} - 1\rangle,$$ and applying Lemma \ref{Wn=Fstar} completes the proof.

\end{proof}


\bibliographystyle{amsplain}


\end{document}